\documentclass[12pt]{amsart}

\usepackage{amsmath,amssymb,amscd,amsxtra}
\usepackage{latexsym}
\usepackage{graphicx}
\usepackage{hyperref}

\newcommand{\C}{\mathbb C}
\newcommand{\R}{\mathbb R}
\newcommand{\0}{\bf 0}

\newcommand{\ip}[2]{\langle#1,#2\rangle}

\def \Sum{\displaystyle\sum}
\newcommand{\norm}[1]{\left\| #1 \right\|}

\DeclareMathOperator*{\argmin}{arg\,min}

\newtheorem{thm}{Theorem}
\newtheorem{lemma}{Lemma}
\newtheorem{prop}{Proposition}
\newtheorem{cor}{Corollary}
\theoremstyle{remark}
\newtheorem{rem}{Remark}
\theoremstyle{definition}
\newtheorem{deft}{Definition}
\newtheorem{example}{Example}
\newtheorem{conjecture}{Conjecture}

\begin{document}

\title{Optimization methods for frame conditioning and application to graph Laplacian scaling}

\author{Radu Balan}

\address{Radu Balan\\
Department of Mathematics\\
University of Maryland\\
College Park, MD, 20742 USA}

\email{rvbalan@cscamm.umd.edu}

\author{Mathew Begu\'e}

\address{Mathew Begu\'e\\
Department of Mathematics\\
University of Maryland\\
College Park, MD, 20742 USA}

\email{begue@math.umd.edu}

\author{Chae Clark}

\address{Chae Clark\\
Department of Mathematics\\
University of Maryland\\
College Park, MD, 20742 USA}

\email{cclark18@math.umd.edu}

\author{Kasso A.~Okoudjou}

\address{Kasso A.~Okoudjou\\
Department of Mathematics\\
University of Maryland\\
College Park, MD, 20742 USA}

\email{kasso@math.umd.edu}


\date{\today}


\begin{abstract} 
A frame is scalable if each of its vectors can be rescaled in such a way that the resulting set becomes a Parseval frame. In this paper, we consider four different optimization problems for determining if a frame is scalable. We offer some algorithms to solve these problems. We then apply  and extend our methods to the problem of reweighing (finite) graph so as to minimize the condition number of the resulting Laplacian. 
\end{abstract}

\maketitle \pagestyle{myheadings} \thispagestyle{plain}
\markboth{R.BALAN, M.BEGU\'E, C. CLARK, AND K.OKOUDJOU}{OPTIMIZATION METHODS FOR FRAME CONDITIONING}

\section{Introduction}
The notion of scalable frame has been investigated in recent years \cite{CKLMNarayanS14, KOscalable2014, cahill2013, KOpreconditioning}, where the focus was more on characterizing frames whose vectors can be rescaled resulting in a tight frame. For completeness, we recall that a set of vectors $F=\{f_i\}_{i=1}^M$ in some (finite dimensional) Hilbert space $\mathcal H$  is a frame for $\mathcal{H}$ if there exist two constants $0<A\leq B<\infty$ such that 

$$A\|x\|^2 \leq \sum_{i=i}^M |\ip{x}{f_i}|^2 \leq B\|x\|^2$$ for all $x\in \mathcal{H}.$ When $A=B$ the frame is said to be tight and if in addition, $A=B=1$ it is termed a Parseval frame. When $F=\{f_i\}_{i=1}^M$ is a frame, we shall abuse notations and denote by $F$ again, the $n\times M$ matrix whose $i^{th}$ column is $f_i$, and where $n$ is the dimension of $\mathcal{H}$. Using this notation, the frame operator is the $n\times n$ matrix $S=FF^{*}$ where $F^{*}$ is the adjoint of $F$. It is a folklore to note that $F$ is a frame if and only if $S$ is a positive definite operator and the optimal lower frame bound, $A$, coincides with the lowest eigenvalue of  $S$ while  the optimal upper frame bound, $B$, equals the largest eigenvalue of $S$. We refer to \cite{CasKut2013, CasazzaFramesChapter, Okou2016} for more details on frame theory. 

It is apparent that tight frames are optimal frames in the sense that the condition number of their frame operator is $1$. We recall that, the \emph{condition number} of a matrix $A$, denoted $\kappa(A)$, is defined as the ratio of the largest singular value and the smallest singular value of $A$, i.e., $\kappa(A)=\sigma_{\max}(A)/\sigma_{\min}(A)$.  By analogy, for a frame in a Hilbert space $\{f_i\}_{i=1}^M\subseteq \mathcal H$ with optimal frame bounds $A$ and $B$, we define the condition number of the frame to be the condition number of its associated frame operator $\kappa(\{f_i\}):=\kappa(S)=B/A$.  In particular, if a frame is Parseval then its condition number equals $1$. In fact, a frame is tight if and only if its condition number is $1$.
Scalable frames were precisely introduced to turn a non optimal (non-tight) frame into an optimal one, by just rescaling the length of each frame vector. More precisely, 

\begin{deft}[{\cite[Definition 2.1]{KOFTscalable}}]
A frame $\{f_i\}_{i=1}^M$ in some Hilbert space $\mathcal H$ is called a \emph{scalable frame} if there exist nonnegative numbers $s_1,...,s_M$ such that $\{s_i f_i\}_{i=1}^M$ is a Parseval frame for $\mathcal H$.
\end{deft}

It follows from the definition that a frame $\{f_i\}_{i=1}^M$ is scalable if and only if there exist scalars $s_i\geq 0$ so that 
$$\kappa\left(\Sum_{i=1}^M s_i^2 f_i f_i^*\right)=1.$$ To date  various equivalent characterizations of  scalable frames have been proved and attempts to measure how close to scalable a non-scalable frame is  have been offered \cite{KOscalable2014, cahill2013, KOpreconditioning, KOPrecondPF2016}.  In particular, if a frame is not scalable, then one can naturally measure how ``not scalable" the frame is by measuring
\begin{equation}\label{distancetoI}
\min_{s_i\geq 0} \norm{I_n-\Sum_{i=1}^M s_i^2 f_i f_i^*}_F,
\end{equation}
as proposed in \cite{ChenKOPWscalable}, where $\norm{\cdot}_F$ denotes the Frobenius norm of a matrix.  Other measures of scalability were also proposed by the same authors. However, it is not clear that, when a frame is not scalable, an optimal solution to~\eqref{distancetoI} yields a frame $\{s_if_i\}$ that is as best conditioned as possible. Recently,  the relationship between the solution to this problem and the condition number of a frame has been investigated in \cite{CasazzaChen2015}. In particular, Casazza and Chen show that the problem of minimizing the condition number of a  scaled frame 

\begin{equation}\label{mincondscf}
\min_{s_i\geq0}\kappa\left(\Sum_{i=1}^M s_i^2 f_i f_i^*\right),
\end{equation}

is equivalent to solving the minimization problem 

\begin{equation}\label{distancetoI_2}
\min_{s_i\geq 0} \norm{I_n-\Sum_{i=1}^M s_i^2 f_i f_i^*}_2,
\end{equation}
where $\norm{\cdot}_2$ is the operator norm of a matrix. Specifically they show that any optimizer of (\ref{mincondscf}) is also an optimizer of (\ref{distancetoI_2}); vice-versa, any optimizer of (\ref{distancetoI_2}) minimizes the condition number in (\ref{mincondscf}).
 Furthermore, they show that the optimal solution to~\eqref{distancetoI}
does not even have to be a frame, and so would yield an undefined condition number for the corresponding system. 

In this chapter, we consider numerical solutions to the scalability problem. Recall that   a frame $F=\{f_i\}_{i=i}^M \subset \mathcal{H}$ is scalable if and only if the exist scalars $\{s_i\}_{i=1}^M\subset [0, \infty)$ such that $$\sum_{i=1}^Ms_i^2f_if_i^{\*}=I.$$ Consequently, the condition number of the scaled frame $\tilde{F}=\{s_if_i\}_{i=i}^M$ is $1$.  We are thus interested in investigating the solutions to the following three optimization problems:

\begin{equation}\label{eq1}
\min_{s_{i}\geq0\,,\,s\neq\0} \dfrac{\lambda_{\max}\left(\sum_{i=1}^{M}s_{i}^{2}f_{i}f_{i}^{*}\right)}{\lambda_{\min}\left(\sum_{i=1}^{M}s_{i}^{2}f_{i}f_{i}^{*}\right)}.
\end{equation}

\begin{equation}\label{eq2}
\min_{\begin{array}{c} \mbox{$s_{i}\geq0\,,\,s\neq\0$} \\ \mbox{$\sum_{i=1}^M s_i^2 \norm{f_i}_2^2 = N$} \end{array} } \lambda_{\max}\left(\sum_{i=1}^{M}s_{i}^{2}f_{i}f_{i}^{*}\right) - \lambda_{\min}\left(\sum_{i=1}^{M}s_{i}^{2}f_{i}f_{i}^{*}\right).
\end{equation}

\begin{equation}\label{eq3}
\min_{s_{i}\geq0\,,\,s\neq\0} \left\|I_{N} - \sum_{i=1}^{M}s_{i}^{2}f_{i}f_{i}^{*}\right\|_{F}.
\end{equation}

Our motivation stems from the fact it appears from the existing literature on scalable frames that the set of all such frames is relatively small, e.g., see \cite{KOscalable2014}. As a result, one is interested in scaling a frame in an optimal manner. For example, by minimizing the condition number of the scaled frame~\eqref{eq1}, or the gap of the spectrum of the scaled frame~\eqref{eq2}. Furthermore, one can try to find the relationship between the optimal solutions to these two problems with the measures of scalability introduced in \cite{ChenKOPWscalable}, of which~\eqref{distancetoI} is a typical example. 

In addition, we investigate these optimization problems from a practical point of view: the existence of fast algorithms to produce optimal solutions. As such, we are naturally lead to consider these problems in the context of convex optimization. We recall that in such a setting one wants to solve for 
$s^*=\argmin_{s} f(s)$ for a real convex function $f:X\to\R\cup\{\infty\}$ defined on a convex set $X$.  Using the  convexity of $f$ and $X$ it follows that:
\begin{enumerate}
\item If $s^*$ is a local minimum of $f$, then it is a global minimum.
\item The set of all (global) minima is convex.
\item  If $f$ is a strictly convex function and a minimum exists, then the minimum is unique.
\end{enumerate}
In addition, the convexity of $f$ and $X$ allows the use of convex analysis to produce fast, efficient algorithmic solvers, we refer to \cite{ConvexBook} and references therein for more details. 

We point out that \eqref{eq1} is equivalent to \eqref{mincondscf} simply by the definition of condition number of a frame.  However, the condition number function  $\kappa$, is not convex. As such,  it is nontrivial to find the optimal solution of~\eqref{eq1}. However, $\kappa$ is a quasiconvex function (see \cite[Theorem 13.6]{Axelsson1996} for a proof), meaning that its lower level sets form convex sets; that is, the set $\{X: \kappa(X)\leq a\}$ forms a convex set for any real $a\geq 0$.  See \cite{eppstein2005} and references therein for a survey on some algorithms that can numerically solve certain quasiconvex problems.
We refer to \cite{Ye2009} for a survey of results on optimizing the condition number.  But we note that, while minimizing the condition number $\kappa$ is not a convex problem, an equivalent convex problem was considered in \cite{Lu2011}. For comparison and completeness we state one of the main results of \cite{Lu2011}. First, observe that if $X$ is a symmetric positive semidefinite matrix, then its \emph{condition number} is defined as
$$\kappa(X)=\left\{\begin{array}{ll} \lambda_{\max}(X)/\lambda_{\min}(X)& \text{if }\lambda_{\min}(X)>0,\\
\infty& \text{if }\lambda_{\min}(X)=0\text{ and }\lambda_{\max}(X)>0,\\
0&\text{if }X\equiv 0.\end{array}\right.$$ 
In this setting, it was proved in  \cite{Lu2011} that the problem of minimizing the condition number is equivalent to solving another problem with convex programming.
\begin{thm}[\cite{Lu2011}, Theorem 3.1]\label{th:Lu2011}
Let $\Omega\subseteq \mathcal S^N$ be some nonempty closed convex subset of $\mathcal S^N$, the space of $N\times N$ symmetric matrices and let $\mathcal S_+^N$ be the space of symmetric positive semidefinite $N\times N$ matrices.  Then the problem of solving
$$\kappa^* = \inf\{\kappa(X):X\in \mathcal S_+^N \cap \Omega\}$$ 
is equivalent to the problem of solving
\begin{equation}\lambda^*=\inf\{\lambda_{\max}(X):X\in t\Omega, \, t\geq 0,\, X\succeq I\},\label{condprob2}\end{equation}
that is, $\lambda^*=\kappa^*$.
\end{thm}

The problem described by \eqref{condprob2} can be restated as solving for optimal scalars $\{s_i\}$ satisfying
\begin{equation}\label{eq4}
\min_{s_{i}\geq0\,,\,s\neq\0} \left\{ \lambda_{\max}\left(\sum_{i=1}^{M}s_{i}^{2}f_{i}f_{i}^{*}\right) \,\left|\, \lambda_{\min}\left(\sum_{i=1}^{M}s_{i}^{2}f_{i}f_{i}^{*}\right.\right)\geq1\right\}.
\end{equation}
Therefore, when we obtain numerical solutions to the condition number problem \eqref{eq1}, we actually solve \eqref{eq4} and the theory of \cite{Ye2009} guarantees that the optimal solutions to both problems are indeed equal.

Theorem \ref{th:Lu2011} has an intuitive interpretation.  Suppose $\kappa(X)=\kappa^*$.  Then rescaling $X$ by a positive scalar, $t$, will also scale its eigenvalues by the same factor $1/t$, thus leaving its condition number, $\kappa(X/t)$, unchanged.  Therefore, without loss of generality, we can assume that $X$ is rescaled so that $\lambda_{\min}(X/t)\geq 1$ which is imposed in the last condition of \eqref{condprob2}.  Once we know that $\lambda_{\min}(X/t)$ is at least 1 then minimizing the condition number of $X/t$ is equivalent to minimizing $\lambda_{\max}(X/t)$ so long as $X/t\in\Omega$ which is guaranteed by the first condition in \eqref{condprob2}.

The goal of this chapter is to investigate the relationship among the solutions to each of the optimization problems~\eqref{eq1}, \eqref{eq2}, and \eqref{eq3}.  In addition, we shall investigate the behavior of the optimal solution to each of these problems vis-\'a-vis the projection of a non-scalable frame onto the set of scalable frames. We shall also describe a number of algorithms to solve some of these problems and compare some of the performances of these algorithms. Finally, we shall apply some of the results of frame scalability to the problem of reweighing a graph in a such a way that the condition number of the resulting Laplacian is as small as possible. The chapter is organized as follow. In Section~\ref{sec2} we investigate the three problems stated above and compare their solutions, and in Section~\ref{sec3} we consider the application to finite graph reweighing.

\section{Non-scalable frames and optimally conditioned scaled frames}\label{sec2}We begin by showing the relationship between the three formulations of this scalability problem.
We shall first show the equivalence of these problems when a frame is exactly scalable, and present toy examples of the different solutions obtained when a frame is only approximately scalable.

\begin{lemma}\label{lem1}
Let $F=\{f_i\}_{i=1}^M$ be a frame in $\R^N$. Then the following statements are equivalent:
\begin{enumerate}
\item[(a)]  $F = \{f_{i}\}_{i=1}^{M}$ is a scalable frame.
\item[(b)]  Problem~\eqref{eq1} has a global minimum solution, $s^{*}=\{s_{i}^{*}\}$, with objective function value 1.
\item[(c)] Problem~\eqref{eq2} has a global minimum solution, $s^{*}=\{s_{i}^{*}\}$, with objective function value 0.
\item[(d)] Problem~\eqref{eq3} has a global minimum solution, $s^{*}=\{s_{i}^{*}\}$, with objective function value 0.
\end{enumerate}
\end{lemma}

\begin{proof}
Assume $F$ is scalable with weights, $\{s_{i}\}_{i=1}^{M}$.
Then $\widetilde{S} = \sum_{i=1}^{M}s_{i}^{2}f_{i}f_{i}^{*} = I_{N}$, and the largest and smallest eigenvalue of the scaled frame operator is 1,
\begin{equation*}
\dfrac{\lambda_{\max}\left(\sum_{i=1}^{M}s_{i}^{2}f_{i}f_{i}^{*}\right)}{\lambda_{\min}\left(\sum_{i=1}^{M}s_{i}^{2}f_{i}f_{i}^{*}\right)} = \dfrac{\lambda_{\max}\left(\widetilde{S}\right)}{\lambda_{\min}\left(\widetilde{S}\right)} = 1.
\end{equation*}

Assume problem~\eqref{eq1} has a global minimum solution, $\{s_{i}\}_{i=1}^{M}$.
As, $\lambda_{\max} \geq \lambda_{\min}$, the feasible solution must result in $\lambda_{\max} = \lambda_{\min} = A$.
Applying this feasible solution as a scaling of $F$, we have,
\begin{equation*}
\widetilde{S} = \sum_{i=1}^{M}s_{i}^{2}f_{i}f_{i}^{*} = AI_{N}.
\end{equation*}
By normalizing the feasible solution by the square-root of $A$, we have the Parseval scaling,
$$\{\tilde{s}_{i}\}_{i=1}^{M} = \left\{\dfrac{1}{\sqrt{A}}s_{i}\right\}_{i=1}^{M}.$$
 We have just proved that (a) and (b) are equivalent.

Assume $F$ is scalable with weights, $\{s_{i}\}_{i=1}^{M}$.
Then $\widetilde{S} = \sum_{i=1}^{M}s_{i}^{2}f_{i}f_{i}^{*} = I_{N}$, and the difference between the largest and smallest eigenvalue of the scaled frame operator is 0,
\begin{equation*}
\lambda_{\max}\left(\sum_{i=1}^{M}s_{i}^{2}f_{i}f_{i}^{*}\right) - \lambda_{\min}\left(\sum_{i=1}^{M}s_{i}^{2}f_{i}f_{i}^{*}\right) = \lambda_{\max}\left(\widetilde{S}\right) -\lambda_{\min}\left(\widetilde{S}\right) = 0.
\end{equation*}
Additionally $N=tr(I_N) = \sum_{i=1}^M s_i^2 \norm{f_i}_2^2$ which shows that $\{s_i\}_{i=1}^M$ is a feasible solution for (\ref{eq2}).

Assume problem~\eqref{eq2} has a global minimum solution, $\{s_{i}\}_{i=1}^{M}$.
As, $\lambda_{\max} \geq \lambda_{\min}$, the feasible solution must result in $\lambda_{\max} = \lambda_{\min} = A$.
Applying this feasible solution as a scaling of $F$, we have,
\begin{equation*}
\widetilde{S} = \sum_{i=1}^{M}s_{i}^{2}f_{i}f_{i}^{*} = AI_{N}.
\end{equation*}
But the feasibility condition $\sum_{i=1}^M s_i^2 \norm{f_i}_2^2 = N$ implies $N=tr(AI_N)$, hence $A=1$. 
We have just proved that (a) and (c) are equivalent.

Assume $F$ is scalable with weights, $\{s_{i}\}_{i=1}^{M}$.
Then $\widetilde{S} = \sum_{i=1}^{M}s_{i}^{2}f_{i}f_{i}^{*} = I_{N}$, and the objective function for \eqref{eq3} attains the global minimum ,
\begin{equation*}
\left\|I_{N} - \sum_{i=1}^{M}s_{i}^{2}f_{i}f_{i}^{*}\right\|_{F} = \left\|I_{N} - I_{N}\right\|_{F} = 0.
\end{equation*}

Assume problem~\eqref{eq3} has a global minimum solution, $\{s_{i}\}_{i=1}^{M}$, which occurs when $\left\|I_{N} - \sum_{i=1}^{M}s_{i}^{2}f_{i}f_{i}^{*}\right\|_{F} = 0$.
This implies that $\widetilde{S} = \sum_{i=1}^{M}s_{i}^{2}f_{i}f_{i}^{*} = I_{N}$, and we have a Parseval scaling.
We have just proved that (a) and (d) are equivalent. 
\end{proof}

\begin{rem}\label{remark1}
Lemma~\ref{lem1} asserts that the problem of finding optimal scalings, $\{s_{i}\}_{i=1}^{M}$, for a given scalable frame $F=\{f_{i}\}_{i=1}^{M}$ is equivalent to finding the absolute minimums of the following optimization problems:
\begin{itemize}
\item $\min_{s_{i}\geq0\,,\,s\neq\0} \dfrac{\lambda_{\max}\left(\sum_{i=1}^{M}s_{i}^{2}f_{i}f_{i}^{*}\right)}{\lambda_{\min}\left(\sum_{i=1}^{M}s_{i}^{2}f_{i}f_{i}^{*}\right)}$

\item $\min_{\begin{array}{c} \mbox{$s_{i}\geq 0\,,\,s\neq\0$} \\ \mbox{$\sum_{i=1}^M s_i^2 \norm{f_1}_2^2 = N$} \end{array} } \lambda_{\max}\left(\sum_{i=1}^{M}s_{i}^{2}f_{i}f_{i}^{*}\right) - \lambda_{\min}\left(\sum_{i=1}^{M}s_{i}^{2}f_{i}f_{i}^{*}\right)$

\item $\min_{s_{i}\geq0\,,\,s\neq\0} \left\|I_{N} - \sum_{i=1}^{M}s_{i}^{2}f_{i}f_{i}^{*}\right\|_{F}$
\end{itemize}

\end{rem}

Lemma~\ref{lem1} is restrictive in that it requires the frame $F=\{f_i\}_{i=1}^M$ be scalable to state equivalence among problems, but there can be a wide variance in the solutions obtained when the frame is not scalable.
Even nearly-tight frames vary in initial feasible solutions.
We briefly consider $\epsilon$-tight frames and analyze the distance from the minimum possible objective function value.

Let $F_{\epsilon}=\{g_{i}\}_{i=1}^{M}$ with $\norm{g_i}_2=1$ for all $i$ be an $\epsilon$-tight frame such that,
$$(1-\epsilon)I_{N} \preceq \sum_{i=1}^{M}g_{i}g_{i}^{*} \preceq (1+\epsilon)I_{N}.$$
First considering the case in which the frame cannot be conditioned any further, so the optimal scaling weights are $s_{i} = 1$.
Analyzing the solution produced by the three optimization methods, we see the difference in solutions produced.
\begin{equation*}\begin{split}
&\dfrac{\lambda_{\max}\left(\sum_{i=1}^{M}s_{i}^{2}g_{i}g_{i}^{*}\right)}{\lambda_{\min}\left(\sum_{i=1}^{M}s_{i}^{2}g_{i}g_{i}^{*}\right)} = \dfrac{\lambda_{\max}\left(\sum_{i=1}^{M}g_{i}g_{i}^{*}\right)}{\lambda_{\min}\left(\sum_{i=1}^{M}g_{i}g_{i}^{*}\right)} = \dfrac{1+\epsilon}{1-\epsilon} = 1 +\dfrac{2\epsilon}{1-\epsilon}.\\
&\lambda_{\max}\left(\sum_{i=1}^{M}s_{i}^{2}f_{i}f_{i}^{*}\right) - \lambda_{\min}\left(\sum_{i=1}^{M}s_{i}^{2}g_{i}g_{i}^{*}\right) = (1+\epsilon) - (1-\epsilon) = 2\epsilon\\
&\lambda_{\max}\left(\sum_{i=1}^{M}s_{i}^{2}g_{i}g_{i}^{*} \right) = \lambda_{\max}\left(\sum_{i=1}^{M}g_{i}g_{i}^{*} \right) = 1 + \epsilon.
\end{split}\end{equation*}
We lack the information necessary to give exact results for formulation \eqref{eq3}, so we instead give an upper bound when $s_i=1$.
\begin{equation*}\begin{split}
\left\|I_{N} - \sum_{i=1}^{M}s_{i}^{2}g_{i}g_{i}^{*}\right\|_{F} &= \left\|I_{N} - \sum_{i=1}^{M}g_{i}g_{i}^{*}\right\|_{F}\\
&\leq \sqrt{N}\left\|I_{N} - \sum_{i=1}^{M}g_{i}g_{i}^{*}\right\|_{2}\\
&\leq \epsilon \sqrt{N}.
\end{split}\end{equation*}

It makes sense that we could enforce this constraint, as we could re-normalize the frame elements by the reciprocal of the smallest eigenvalue of the frame operator.
It is not true, though, that the scalings produced must be the same.
Moreover, when not using the constraint on the smallest eigenvalue, the scalings can vary wildly.

\begin{rem}\label{remark2}
For general frames, the optimization problems (\ref{eq1})-(\ref{eq3}) do not produce tight frames. However they can be solved using special classes of convex optimization algorithms: problems (\ref{eq1}) and (\ref{eq2}) are solved by 
{\em Semi-Definite Programs} (SDP), whereas problem (\ref{eq3}) is solved by a Quadratic Program (QP) -- see \cite{ConvexBook} for details on SDPs and QPs. 
In the following we state these SDPs explicitly. 

{\bf SDP 1} -- Operator Norm Optimization:
\begin{equation} \label{eq:SDP1}
(t^1,s^{(1)}) = argmin_{\begin{array}{c} \mbox{$t,s_1,\ldots,s_M\geq 0$} \\ \mbox{$\sum_{i=1}^M s_i^2 f_i f_i^*-t I_N -I_N \leq 0$} \\ \mbox{$\sum_{i=1}^M s_i^2 f_i f_i^* + t I_N - I_N\geq 0 $} \end{array} } t  
\end{equation}
This SDP implements the optimization problem (\ref{distancetoI_2}). In turn, as showed by Cassaza and Chen in \cite{CasazzaChen2015}, the solution to this problem is also an optimizer of the condition number optimization problem (\ref{eq1}).
Conversely, assume $s^{(*)}$ is a solution of (\ref{eq1}). Let $A=\lambda_{min}(\sum_{i=1}^M s_i^2 f_i f_i^*)$ and $B=\lambda_{max}(\sum_{i=1}^M s_i^2 f_i f_i^*)$. Let $r=\frac{2}{A+B}$. Then $s^{(*)} = (rs_i^2)_{i=1}^M$ is a solution
of (\ref{eq:SDP1}) and the optimum value of the optimization criterion is $t^1 = rB-1=1-rA$.

{\bf SDP 2} -- Minimum Upper Frame Bound Optimization:
\begin{equation} \label{eq:SDP2}
(t^2,s^{(2)}) = argmin_{\begin{array}{c} \mbox{$t,s_1,\ldots,s_M \geq 0$} \\ \mbox{$\sum_{i=1}^M s_i^2 f_i f_i^* -I_N \geq 0$} \\ \mbox{$\sum_{i=1}^M s_i^2 f_i f_i^*  -t I_N\leq 0 $} \end{array} } t  
\end{equation}
This SDP implements the optimization problem (\ref{eq4}) which as previously discussed, also produces the solution $s^{(2)}$ to (\ref{eq1}). Conversely, assume $s^{(*)}$ is a solution of  (\ref{eq1}). 
Let $A=\lambda_{min}(\sum_{i=1}^M s_i^2 f_i f_i^*)$ and $B=\lambda_{max}(\sum_{i=1}^M s_i^2 f_i f_i^*)$. Let $r=\frac{1}{A}$. Then $s^{(*)} = (rs_i^2)_{i=1}^M$ is a solution of (\ref{eq:SDP2}), and the optimum value
of the optimization criterion is $t^2=\frac{B}{A}$.

{\bf SDP 3} -- Spectral Gap Optimization:
\begin{equation} \label{eq:SDP3}
(t^3,v^3,s^{(3)}) = argmin_{\begin{array}{c} \mbox{$t,v,s_1,\ldots,s_M\geq 0$} \\ \mbox{$\sum_{i=1}^M s_i^2 f_i f_i^* -t I_N\leq 0$} \\ \mbox{$\sum_{i=1}^M s_i^2 f_i f_i^* - v I_N\geq 0 $} \\ \mbox{$\sum_{i=1}^M s_i \norm{f_i}_2^2 = N $} \end{array} } t-v  
\end{equation}
This SDP implements the optimization problem (\ref{eq2}). As remarked earlier (\ref{eq2}) is not equivalent to
any of (\ref{distancetoI_2}),(\ref{eq1}) or (\ref{eq4}). A spectral interpretation of these optimization problems is as follows. The SDP 1 (and implicitly (\ref{eq1}) and (\ref{eq4})) scales the frame so that the largest and smallest
eigenvalues of the scaled frame operator are equidistant and closest to value 1. The SDP 3 scales the frame so that the largest and smallest eigenvalues of the scaled frame operator are closest to one another while the average eigenvalue is set to 1. Equivalently, the solution to SDP 3 also minimizes the following criterion:
\[ \frac{\lambda_{max}(\tilde{S})-\lambda_{min}(\tilde{S})}{\frac{1}{N}tr(\tilde{S})} \]
where $\tilde{S}=\sum_{i=1}^M s_i^2 f_i f_i^*$ is the scaled frame operator.

{\bf QP 4} -- Frobenius Norm Optimization:
\begin{equation} \label{eq:SDP3}
s^{(4)} = argmin_{\begin{array}{c} \mbox{$s_1,\ldots,s_M\geq 0$} \end{array} } \sum_{i,j=1}^M s_is_j |\ip{f_i}{f_j}|^2  - 2 \sum_{i=1}^M s_i^2 \norm{f_i}_2^2 + N
\end{equation}
This QP implements the optimization problem (\ref{eq3}).

\end{rem}

\begin{example}
Consider the 5-element frame, $X\subseteq \R^3$, generated such that each coordinate is a random integer from 0 to 5.  
$$X=\left[\begin{array}{ccccc}2 & 4 & 1 & 4 & 4 \\3 & 1 & 2 & 0 & 2 \\1 & 4 & 3 & 5 & 2\end{array}\right]$$

We then numerically compute $X_\kappa$, $X_g$, $X_F$, which are the rescaled frames that minimize problems SDP 1, SDP 3 and QP 4, respectively.  That is, $X_\kappa$ is the rescaled frame, $X_\kappa=\{s_i f_i\}$, such that $s^*=\{s_i\}$ is the minimizer to Problem \eqref{distancetoI_2}, which also minimizes the frame condition number, $\kappa$.  Similarly, $X_g$ is rescaled to minimize the eigenvalue gap $\lambda_{\max{}}-\lambda_{\min{}}$ while the average eigenvalue is 1, and $X_F$ is rescaled to minimize Frobenius distance to the identity matrix.

In our numerical implementation minimizing condition number, we used the CVX toolbox in MATLAB \cite{cvxMATLAB} which is a solver for convex optimization problems.

Let $s_\kappa$, $s_g$, and $s_F$ denote the scaling vectors that determine the frames $X_\kappa$, $X_g$, and $X_F$, respectively.  That is, $X_\kappa=S_\kappa^{1/2} X$ where $S_\kappa$ is the diagonal matrix with values given by $s_\kappa$, and so on.  We obtained scalings

\begin{center}
\begin{tabular}{r l l l l l}
$s_\kappa$=&[0.0187,& 0,& 0.0591,& 0.0122,& 0.0242],\\
$s_g$=&[0.0875,&0,& 0.0398,& 0.0297,& 0],\\
$s_F$=&[0.0520,& 0,& 0.0066,& 0.0177,& 0].
\end{tabular}
\end{center}

The results comparing each of the four frames are summarized in Table \ref{tab:compare}.

\begin{table}
\begin{center}
\begin{tabular}{|c||c|c|c|c|c|c|}\hline
&$\lambda_{\min{}}$&$\lambda_{\max{}}$&$\kappa$&$(\lambda_{\max{}}-\lambda_{\min{}})/\frac{1}{N}\sum_{i=1}^N \lambda_i$ & $\norm{I_3-\cdot}_F$ & $\norm{I_3-\cdot}_2$ \\\hline
$X$&4.1658&110.41&26.504&2.5296&109.95 &109.41 \\\hline
$X_\kappa$&  0.1716 & 1.8284 & {\bf 10.655}& 2.2888 & 1.4348 & {\bf 0.8284} \\\hline
$X_g$ & 0.0856 & 2.3558 & 27.501 & {\bf 2.2701} & 1.6938 & 1.3558 \\\hline
$X_F$ & 0.01672 & 1.1989 & 71.667 & 2.2903 & {\bf 1.2048} & 0.9832 \\\hline
\end{tabular}
\end{center}
\caption{Comparisons of extreme eigenvalues, condition number, relative spectral gap, Frobenius distance to identity, and the operator norm distance to identity for the non scalable frame $X$ and its rescaled versions that minimize Problems \eqref{eq1}--\eqref{eq3}.}\label{tab:compare}
\end{table}

Observe that each of the three methods can produce widely-varying spectra.

\end{example}

We now demonstrate special conditions in which a frame's condition number can be decreased using matrix perturbation theory.
\begin{lemma}[Weyl's Inequality, {\cite[Corollary 4.9]{stewart1990matrix}}]
Let $A$ be a Hermitian matrix with real eigenvalues $\{\lambda_i(A)\}_{i=1}^d$ and let $B$ be a Hermitian matrix of the same size as $A$ with eigenvalues $\{\lambda_i(B)\}_{i=1}^d$.  Then for any $i=1,...,d$ we have 
$$\lambda_i(A+B)\in[\lambda_i(A)+\lambda_1(B),\lambda_i(A)+\lambda_d(B)].$$
\end{lemma}
An immediate corollary of Weyl's inequality tells us that perturbing a matrix by a positive semidefinite matrix will cause the eigenvalues to not decrease.
\begin{cor}\label{cor:Weyl}
Let $A$ be a Hermitian matrix with real eigenvalues $\{\lambda_i(A)\}_{i=1}^d$ and let $B\succeq 0$ be Hermitian and of the same size of $A$.  Then for any $i=1,...,d$, we have $\lambda_i(A)\leq \lambda_i(A+B)$.  The inequality is strict if $B\succ 0$ is positive definite.
\end{cor}

\begin{lemma}\label{lem:eigaddorthog}
Let $ f$ be an eigenvector of $A$ with associated eigenvalue $\lambda$.   Let $B$ be a matrix of the same size as $A$ with the property that $B f=0$.  Then $ f$ is an eigenvector of $A+B$ with eigenvalue $\lambda$.
\end{lemma}

\begin{lemma}[{\cite[Section 1.3]{Tao2012}}]
Let $A$ and $B$ be two $N\times N$ Hermitian matrices of same size. Then for any $i=1,...,N$, the mapping $t\mapsto \lambda_i(A+tB)$ is Lipschitz continuous with Lipschitz constant $\norm{B}_2$.  
\end{lemma}

\begin{cor}\label{cor:Lipschitzeig}
Let $A$ be an $N\times N$ Hermitian matrix with simple spectrum and minimum eigengap $\delta>0$, i.e.,
$$\delta=\min_{i\neq j}|\lambda_i-\lambda_j|.$$
Let $B$ be a non-negative Hermitian matrix of same size as $A$.
Then the mappings $t\mapsto \lambda_i(A+tB)$ are interlacing:
\[ \lambda_1(A)\leq \lambda_1(A+tB) \leq \lambda_2(A)\leq \lambda_2(A+tB)\leq \cdots \leq \lambda_{N-1}(A+tB) \leq \lambda_N(A) \leq \lambda_N(A+tB) \]
for $t\in(0,\frac{\delta}{\norm{B}_2})$. 
\end{cor}

The following theorem gives conditions in which we can guarantee that the condition number of  frame can be reduced.
\begin{thm}
Let $F=\{f_i\}_{i=1}^m\subseteq \C^d$ be a frame that is not tight and whose frame operator has simple spectrum with minimal eigengap $\delta>0$.  Suppose that there exists some index $k$ such that $f_k$ is orthogonal to the eigenspace corresponding to $\lambda_{\max}(FF^*)$ and not orthogonal to the eigenspace corresponding to $\lambda_{\min}(FF^*)$.  Then there exists a rescaled frame $\tilde F=\{s_i f_i\}_{i=1}^m$ satisfying $\kappa(\tilde F)<\kappa(F)$.  In particular, one scaling that decreases the condition number is 
$$s_i=\left\{\begin{array}{l l} 
\frac{m}{m-1+\sqrt{1+\gamma}},& \text{ for }i\neq k\\
\frac{m\sqrt{1+\gamma}}{m-1+\sqrt{1+\gamma}}, & \text{ for }i=k\end{array}\right.$$
for $\gamma\in(0,\delta\norm{f_k}^{-2})$.
\end{thm}

\begin{proof}
Let $f_k$ denote the frame element as described in the assumptions in the statement of the theorem.  For $\gamma\in(0,\delta)$, consider the frame operator $HH^*=FF^*+\gamma f_kf_k^*$ which corresponds to the rescaled frame of $F$ where each scale $s_i=1$ except for $s_k=\sqrt{1+\gamma}$.   The matrix $f_kf_k^*$ is Hermitian and positive semidefinite so by Corollary \ref{cor:Weyl}, we have $\lambda_i(FF^*)\leq \lambda_i(HH^*)$ for every $i=1,...,N$.
Then by Corollary \ref{cor:Lipschitzeig}, the eigenvalues of the frame operator $HH^*$ satisfies the following interlacing property:
$$\lambda_1(FF^*)\leq \lambda_1(HH^*)\leq \lambda_2(FF^*) \leq \lambda_2(HH^*)\leq \cdots\leq \lambda_N(FF^*)= \lambda_N(HH^*),$$
where the last equality follows from Lemma \ref{lem:eigaddorthog} and the fact that $f_k$ is orthogonal to the eigenspace corresponding to $\lambda_N(FF^*)$.

We can now compute
$$\kappa(FF^*)=\frac{\lambda_N(FF^*)}{\lambda_1(FF^*)} \geq \frac{\lambda_N(HH^*)}{\lambda_1(HH^*)} = \kappa(HH^*).$$

Finally, we renormalize the scales $\{s_i\}$ by the constant factor $m(m-1+\sqrt{1+\gamma})^{-1}$ to preserve the property that $\sum_{i=1}^m s_i =m$.  This renormalization scales all eigenvalues by the same factor which leaves the condition number unchanged.  The frame $$\tilde F = \frac{m}{m-1+\sqrt{1+\gamma}} H$$ is the frame described in the statement of the theorem, which concludes the proof.
\end{proof}

\begin{rem}
Having discussed the equivalence between the formulations above, we have seen that they do not necessarily produce similar solutions.
This brings the question of which formulation we should use in general, to the forefront. One could answer this question by seeking a metric that  best describes the distance of a frame to  the set of tight frames.  
This is similar to the Paulsen problem \cite{CahillCasazza13}, in that, after we have solved one of the formulations above, we produce a scaling and subsequent new frame and wish to determine the distance of this new frame to the canonical  Parseval frame associated to our original frame.
In \cite{ChenKOPWscalable}, the question of distance to Parseval frames was generalized to include frames that could be made tight with a diagonal scaling, resulting in the distance between a frame and the set of scalable frames: 
\begin{equation}\label{sec3_eq1}
d_{F} = \min_{\Psi\in\mathcal{SC}(M,N)} \|F - \Psi\|_{F}.
\end{equation}
However, due to the fact that the topology of the set of scalable frames $\mathcal{SC}(M,N)$ is not yet well-understood, computing $d_F$ is almost impossible for a non-scalable frame. A source of future work involves finding bound on $d_F$ using the optimal solutions to the three problems we stated above to analyze and produce bounds on the minimum distance.
\end{rem}

\section{Minimizing condition number of graphs}\label{sec3}
In this section we outline how to apply and generalize the problems the optimization problems from Section \ref{sec2} in the setting of (finite) graph Laplacians.  This task is not a simply as directly applying the condition number minimization problem \eqref{eq1}, and the others, with graph Laplacian operators.  

Recall that any finite graph has a corresponding positive semidefinite Laplacian matrix with eigenvalues $\{\lambda_k\}_{k=0}^{N-1}$ and eigenvectors $\{ f_k\}_{k=0}^{N-1}$.  Further any graph has smallest eigenvalue $\lambda=0$ with multiplicity equal to number of connected components in the graph with eigenvalues equal to constant functions supported on those connected components.  Because any Laplacian's smallest eigenvalue equals 0, its condition number $\kappa(L)$ is undefined.  For simplicity, let us assume that all graphs in this section are connected and hence $0=\lambda_0<\lambda_1\leq\lambda_2\leq \cdots \leq \lambda_{N-1}$.  Suppose we restricted the Laplacian operator to the $(N-1)$-dimensional space spanned by the eigenvectors $ f_1,..., f_{N-1}$.  Then this new operator, call it $L_0$, has eigenvalues $\lambda_1,...,\lambda_{N-1}$ which are all strictly positive.  Now, $\kappa(L_0)$, the condition number of $L_0$ is a well-defined number.

Recall that the complete graph on $N$ vertices, $K_N$, is the most connected a graph on $N$ vertices can be since one can traverse from any two vertices on precisely one edge.  It is the only graph that has all nonzero eigenvalues equal, i.e., $\lambda_0=0$ and $\lambda_1=\lambda_2=\cdots=\lambda_{N-1}=N-1$.  This graph achieves the highest possible algebraic connectivity, $\lambda_1$, of a graph on $N$ vertices.  If we create $L_0$ by projecting the Laplacian of $K_N$ onto the $N-1$-dimensional space spanned by the eigenvectors corresponding with nonzero eigenvalue then $L_0$ equals $N I_{N-1}$, that is a the $(N-1)\times (N-1)$ identity matrix times $N$.
\begin{lemma}
Let $G$ be a connected graph with eigenvalues $\{\lambda_k\}_{k=0}^{N-1}$ and eigenvectors $\{ f_k\}_{k=0}^{N-1}$ of the graph Laplacian $L$.  Let
$\tilde{F}=[ f_1\,  f_2\, \cdots\,  f_{N-1}]$ be the $N\times (N-1)$ matrix of eigenvectors excluding the constant vector $ f_0$.  Then the $(N-1)\times(N-1)$ matrix 
\begin{equation}L_0=\tilde F^* L \tilde F\label{L0def}\end{equation} has eigenvalues $\{\lambda_k\}_{k=1}^{N-1}$ and associated orthonormal eigenvectors $\{\tilde F^* f_k\}_{k=1}^{N-1}$.
\end{lemma}
\begin{proof}
We first show that $\{\tilde F^* f_k\}_{k=1}^{N-1}$ are eigenvectors to $L_0$ with eigenvalues $\lambda_k$.  For any $k=1,...,N-1$ we have 
$$L_0\tilde F^* f_k=\tilde F^*L\tilde F\tilde F^* f_k.$$
But since $\tilde F$ is an orthonormal basis for the eigenspace that its vectors span, then $\tilde F\tilde F^*$ is simply the orthogonal projection onto the eigenspace spanned by $\{ f_1,..., f_{N-1}\}$.  That is, for any vector $f$, we have $\tilde F\tilde F^*f=f-\langle f, f_0\rangle f_0$, which is simply the function $f$ minus its mean value.  For each $k=1,...,N-1$, the eigenvectors $ f_k$ have zero mean, i.e., $\langle  f_k, f_0\rangle=0$.  Hence $\tilde F\tilde F^* f_k= f_k$ and therefore
$$L_0\tilde F^* f_k=\tilde F^*L f_k=\tilde F^*(\lambda_k f_k)=\lambda_k\tilde F^* f_k.$$

The orthonormality of the eigenvectors $\{\tilde F^* f_k\}_{k=1}^{N-1}$ follows directly from the orthonormality of $\{ f_k\}_{k=0}^{N-1}$ and the computation
$$\langle \tilde F^* f_k, \tilde F^* f_j\rangle=(\tilde F^* f_k)^*\tilde F^* f_j= f_k^*\tilde F\tilde F^* f_j= f_k^* f_j=\delta(k,j).$$
\end{proof}
Unlike the Laplacian, the operator in \eqref{L0def} is full rank and its rank equals the rank of the Laplacian.  We denote it $L_0$ because it behaves as the Laplacian after the projection of the function onto the zero'th eigenspace is removed.

For a general finite graph, the Laplacian can be written as the sum of rank-one matrices $L=\sum_{i=1}^m v_i v_i^*$ where $v_i$ is the $i$'th column in the incidence matrix $B$ associated to the $i$'th edge in the graph and $m$ is the total number of edges in the graph.  Thus, the Laplacian can be formed by the product $L=BB^*$.  The columns of the incidence matrix, $B$, as vectors in $\mathbb{\R}^N$ do not form a frame; $B$ has rank $N-1$.  However, the restriction $B$ to the $(N-1)$-dimensional space spanned by $ f_1,...., f_{N-1}$, call it $B_0$, is a frame in that space.  Then the methods of Section \ref{sec2} do apply to the frame $B_0$ with corresponding frame operator $L_0=B_0B_0^*$.  Therefore the operator $L_0$ can also be written as one matrix multiplication $L_0=(\tilde F^* B)(\tilde F^* B)^*$.  For other related results on graphs and frames  we refer to \cite{PenI08}. 
 
We seek scalars $s_i\geq 1$ so that the rescaled frame $\{s_i\tilde F^*v_i\}_{i=1}^m$ is tight or as close to tight as possible.  In terms of matrices, we seek a nonnegative diagonal matrix $X=\operatorname{diag}(s_i)$ so that $\tilde L_0:=\tilde F^*BX^2B^*\tilde F$ has minimal condition number.  The resulting graph Laplacian, denoted $\tilde L_\kappa = B X^2 B^*$, is the operator with minimal condition number, $\tilde L_0$, without the projection onto $(N-1)$ eigenspaces, thus acting on the entire $N$-dimensional space.  One can interpret this problem as rescaling weights of graph edges to not only make $\tilde{L_0}$ as close as possible to the $(N-1)$-identity matrix but also make the $N\times N$ Laplacian, $\tilde L$, as close to possible as the Laplacian of the complete graph $K_N$.

We present the pseudocode for the algorithm, \url{GraphCondition}, that produces $\tilde L_\kappa$, the Laplacian of the graph that minimizes the condition number of $L$.  
\\
\fbox{%
\parbox{\textwidth}{
$L_\kappa$=\url{GraphCondition}$(L, F,B)$\\
where $L$ is the Laplacian matrix of the graph $G$,\\
$ F$ is the $N\times N$ eigenvector matrix of $L$\\
$B$ is the incidence matrix of $L$.

\begin{enumerate}
\item Set $\tilde F= F(:,2:N)$.
\item Use \url{cvx} to solve for $X$ that minimizes $\lambda_{\max}(\tilde F^*BX^2B^*\tilde F)$. \\
subject to:  $X\succeq0$ is diagonal, $\operatorname{trace}(X)\geq t\geq 0$, and $\tilde F^*BX^2B^*\tilde F\succeq I$.
\item Create $L_\kappa= BX^2B^*$.
\end{enumerate}
}
}

\begin{example}\label{ex:cluster}
We consider the barbell graph $G$ which consists of two complete graphs on 5 vertices that are connected by exactly one edge.  The Laplacian for $G$ has eigenvalues $\lambda_1\approx 0.2984$ and $\lambda_{9}\approx 6.7016$, thus giving a condition number of $\kappa(G)\approx 22.45$.  We rescale the edges via the \url{GraphCondition} algoritihm and obtained a rescaled weighted graph $\tilde G_\kappa$ which has eigenvalues $\lambda_1\approx 0.3900$ and $\lambda_{10}\approx 6.991$, thus giving a condition number $\kappa(\tilde G_\kappa)\approx17.9443$.  

Both graphs, $G$ and $\tilde G_\kappa$, are shown in Figure \ref{fig:completeclustercondition}.  The edge bridging the two complete clusters is assigned the highest weight of 1.8473.  All other edges eminating from those two vertices are assigned the smallest weights of 0.7389.  All other edges not connected to either of the two ``bridge" vertices are assigned a weight of 1.1019.

\begin{figure}[htbp]
\begin{center}
\includegraphics[scale=.5]{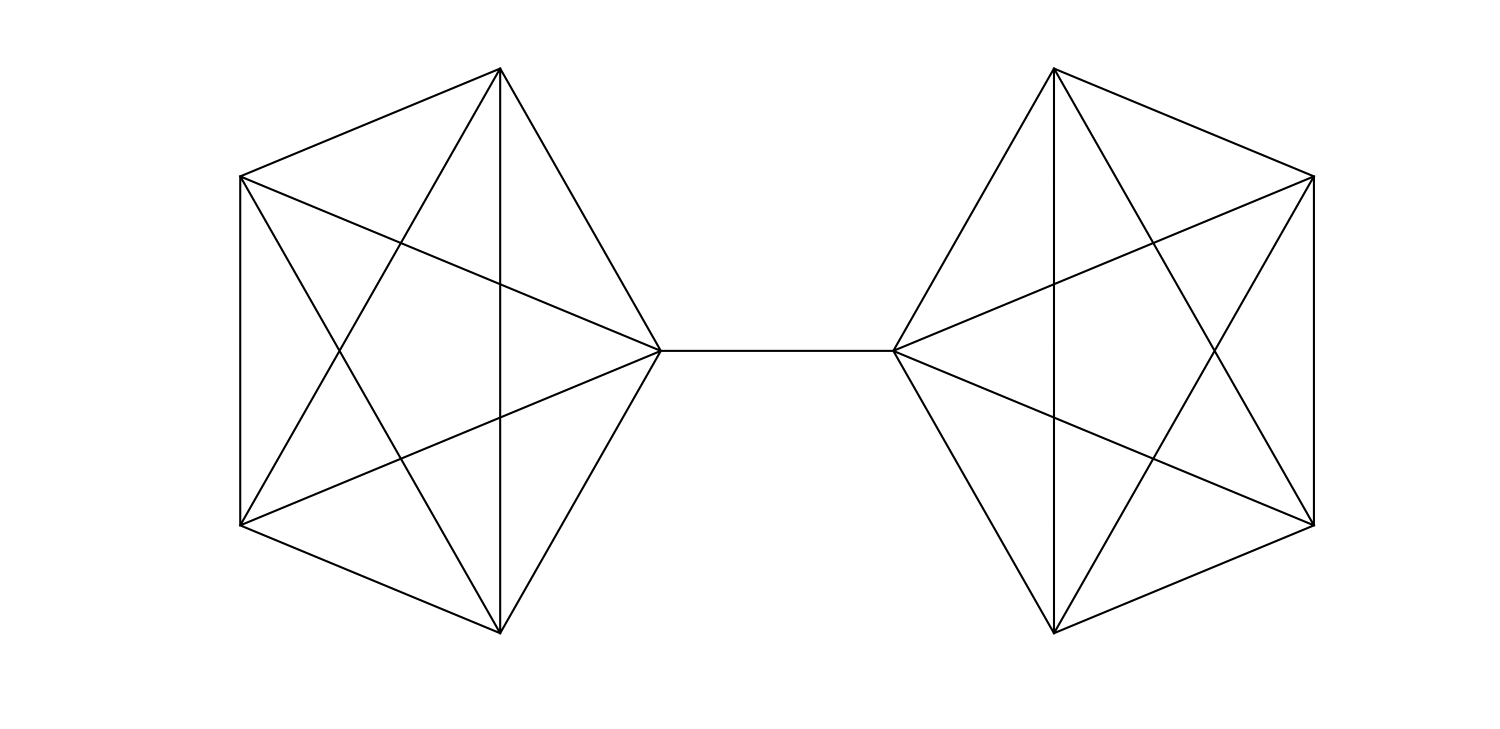}\\
\includegraphics[scale=.5]{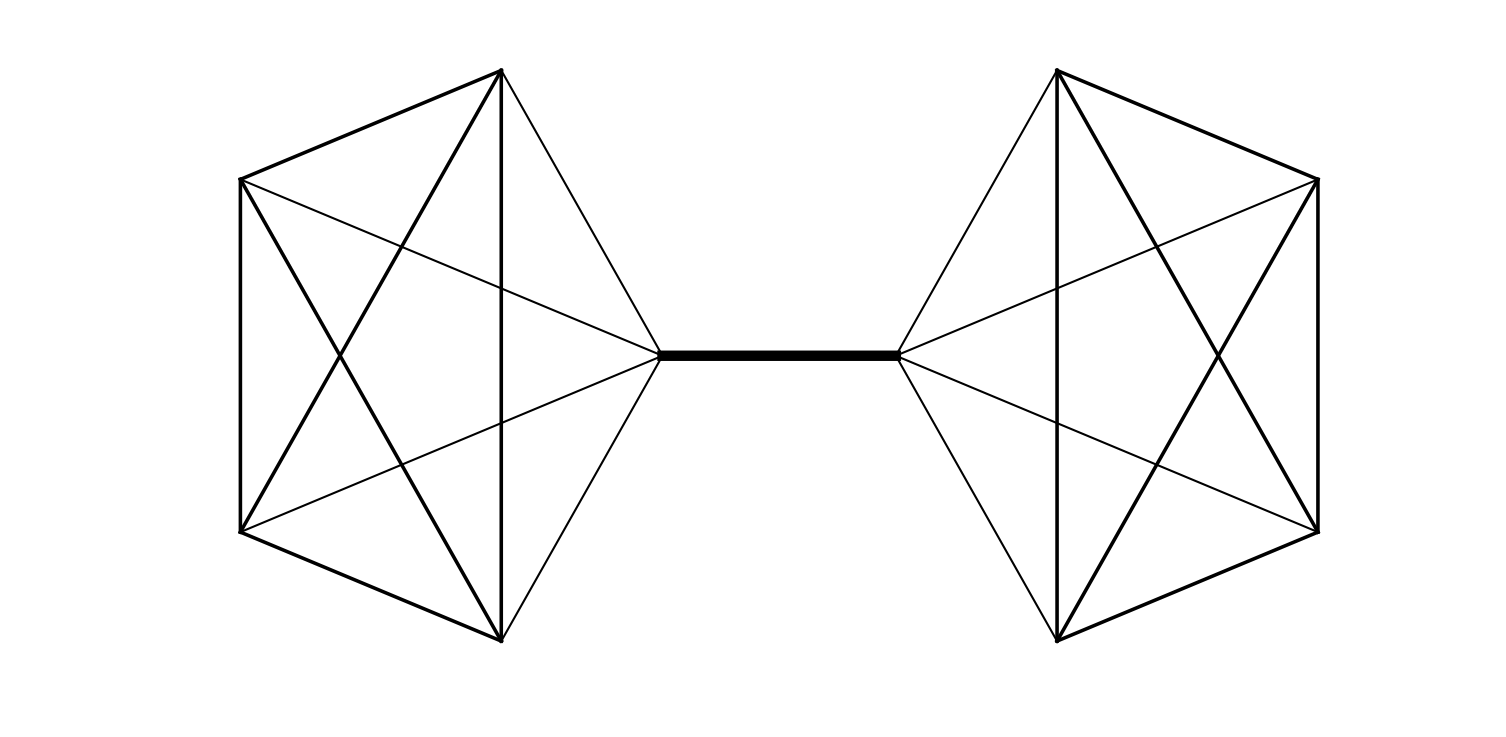}
\end{center}
\caption{Top: The barbell graph $G$.  Bottom: The condiitoned graph with rescaled weights that minimizes the condition number. The width of the edges are drawn to be proportional to the weight assigned to that edge.}\label{fig:completeclustercondition}
\end{figure}
\end{example}

We show in the following example that the scaling coefficients $\{s_i\}_{i=1}^m$ that minimize the condition number of a graph are not necessarily unique.
\begin{example}
Consider the graph $G$ complete graph on four nodes with the edge $(3,4)$ removed.  Then $G$ was rescaled and conditioned via \url{GraphCondition}; both graphs are shown in Figure \ref{fig:conditionnotunique}.  The orignal Laplacian, $L$, and the rescaled conditioned Laplacian, $\tilde L_\kappa$, produced by the \url{GraphCondition} algorithm are given as
$$L=\left[\begin{array}{cccc}3 & -1 & -1 & -1 \\-1 & 3 & -1 & -1 \\-1 & -1 & 2 & 0 \\-1 & -1 & 0 & 2\end{array}\right],\quad
\tilde L_\kappa\approx\left[\begin{array}{cccc} 2.8406 & -0.6812 & -1.0797 & -1.0797 \\-0.6812 & 2.8406 & -1.0797 & -1.0797 \\-1.0797 & -1.0797 & 2.1594 & 0 \\-1.0797 & -1.0797 & 0 & 2.1594\end{array}\right],$$
with spectra
$$\sigma(L)=\{0,2,4,4\},\quad \sigma(\tilde L_\kappa)=\{0,2.1594, 3.5218,4.3188\}.$$
Both Laplacians have a condition number $\kappa(L)=\kappa(\tilde L_\kappa)=2$ which shows that the scaling of edges that minimize condition number are not necessarily unique. 

\begin{figure}[htbp]
\begin{center}
\includegraphics[scale=.33]{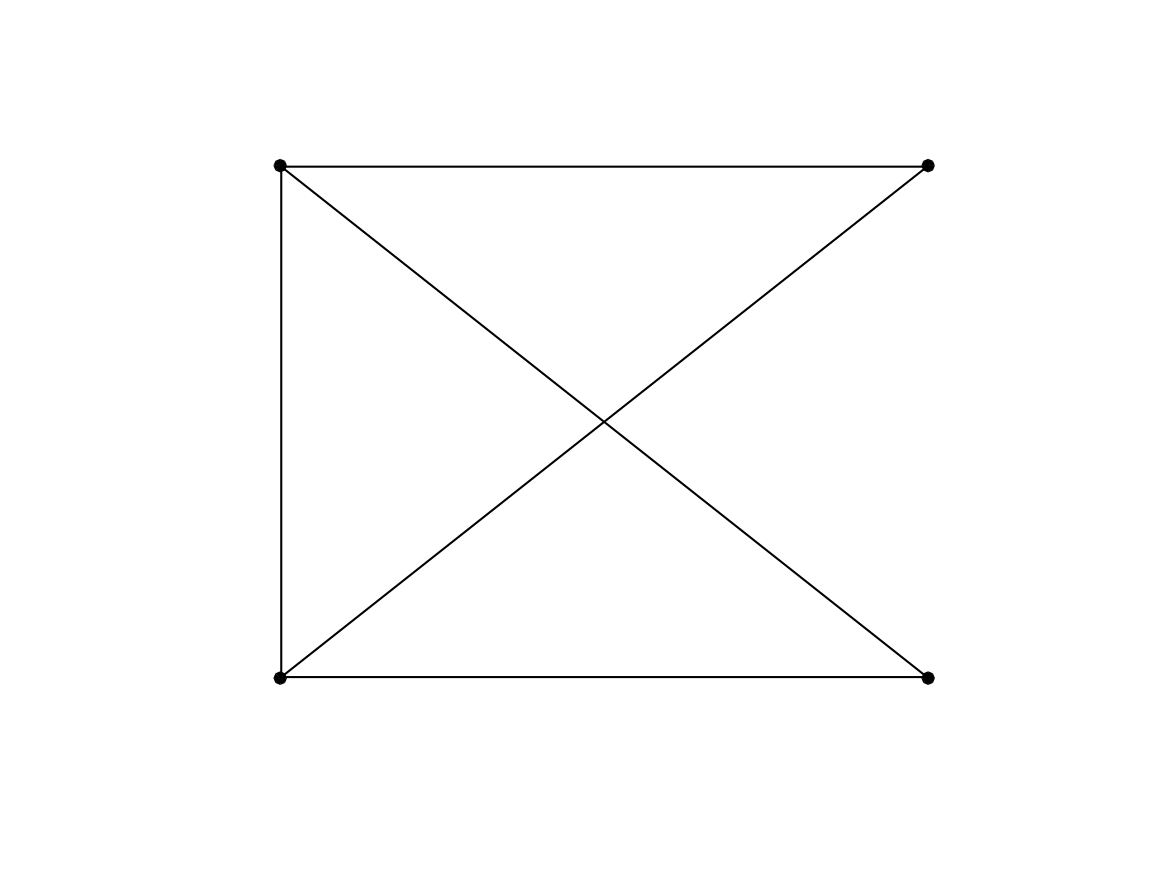}\includegraphics[scale=.33]{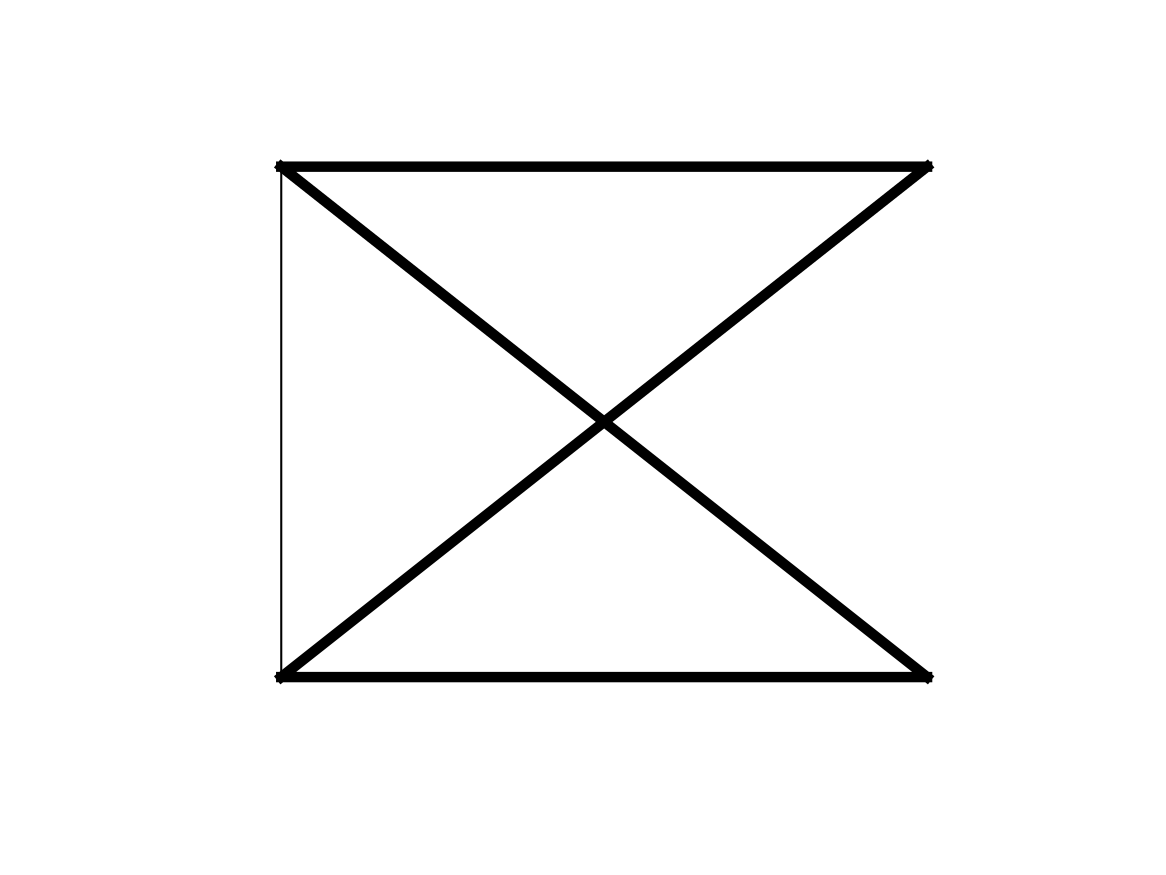}
\end{center}
\caption{The unweighted graph $G$ (left) and its rescaled version $\tilde G_\kappa$ (right) yet both graphs have a condition number equal to 2.}\label{fig:conditionnotunique}
\end{figure}
\end{example}

We prove that the \url{GraphCondition} algorithm will not disconnect a connected graph.
\begin{prop}
Let $G=G(V,E,\omega)$ be a connected graph and let $\tilde G_\kappa=\tilde G_\kappa(V,\tilde E, \tilde \omega)$ be the rescaled version of $G$ that minimizes graph condition number.  Then $\tilde G_\kappa$ is also a connected graph.
\end{prop}
\begin{proof}
Let $\kappa_0:=\kappa(G)\geq 1$ and suppose that $\tilde G_\kappa$ is disconnected.  This implies that $\tilde G_\kappa$ has eigenvalue 0 with multiplicity at least 2 (one for each of its connected components).  This violates the condition $\tilde F^*BX^2B^*\tilde F\succeq I$ in the \url{GraphCondition} algorithm, which yields the unique minimizer.
\end{proof}

We next consider the analogue of minimizing the spectral gap, $\lambda_{N-1}-\lambda_1$, for graphs.  Just as before with condition number, we create the positive definite matrix $L_0$ and its incidence matrix, $B_0$, and minimize its spectral gap by the methods in Section \ref{sec2} to minimize problem \eqref{eq2}.  We denote the rescaled graph that minimizes the spectral gap by $\tilde G_g$.

\begin{example}
We present numerical results of each of the graph rescaling techniques for the barbell graph shown in Figure \ref{fig:completeclustercondition}.  Each of the rescaled graphs are pictured in Figure \ref{fig:clustergraphs} and numerical data is summarized in Table \ref{tab:clusterdata}.

\begin{figure}[htbp]
\begin{center}
\includegraphics[scale=.35]{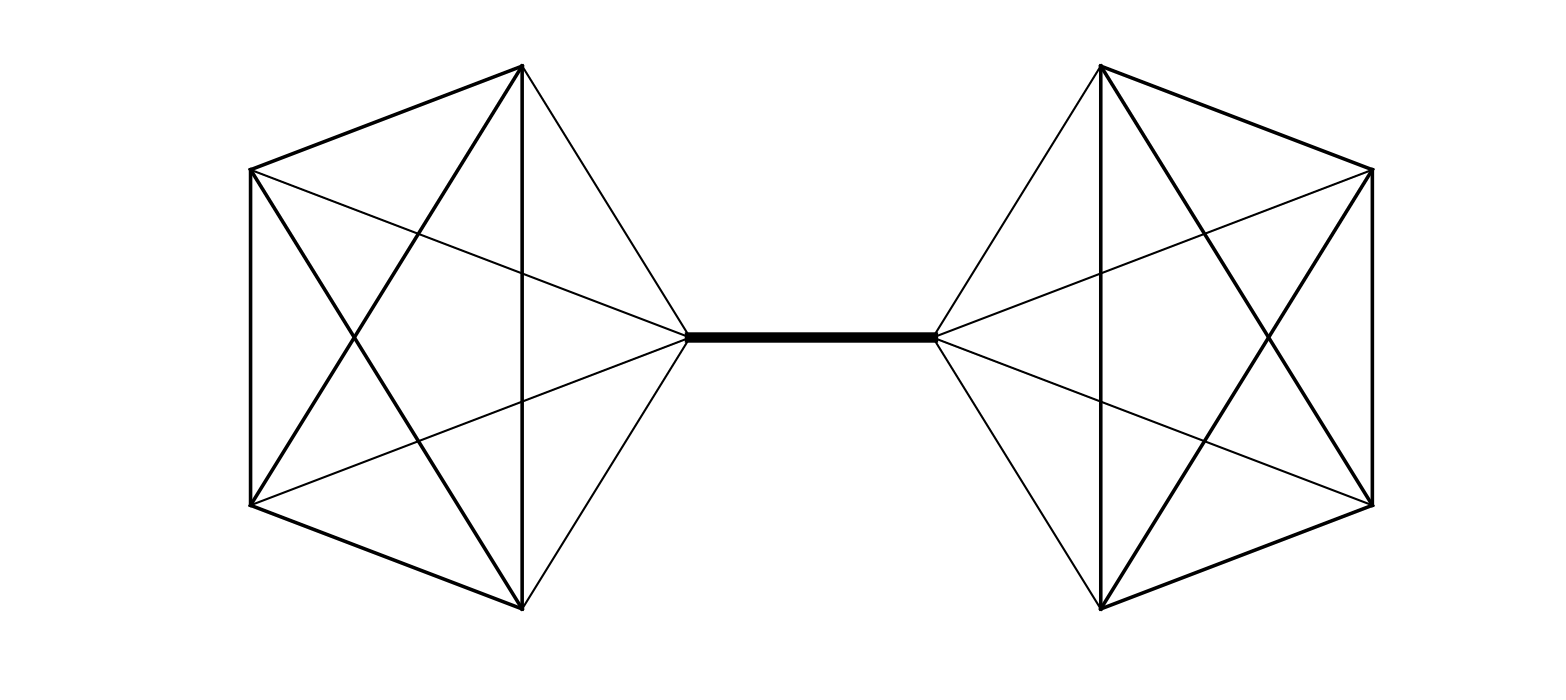}
\includegraphics[scale=.35]{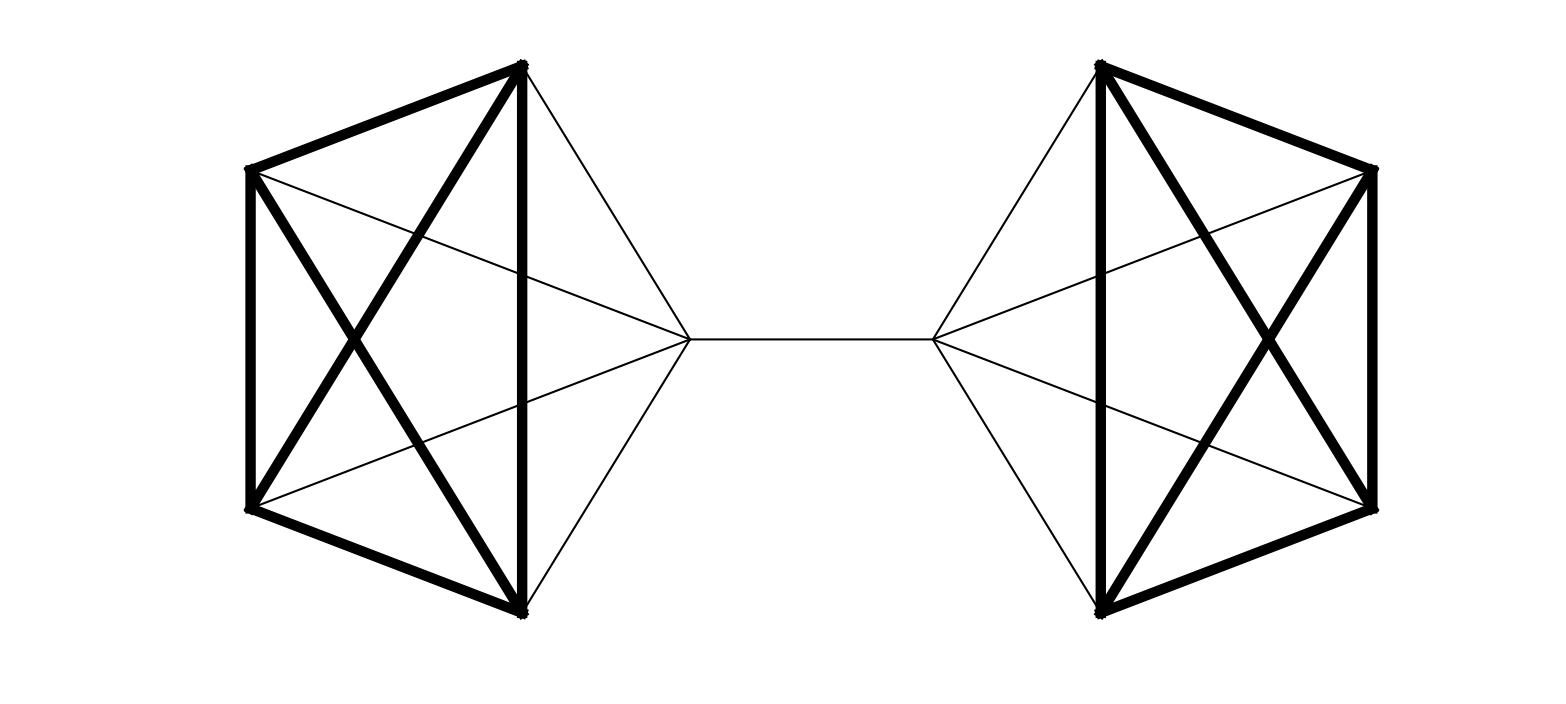}\\
\end{center}
\caption{From top to bottom:  $\tilde G_\kappa$ and $\tilde G_g$, which minimize the condition number and spectral gap, respectively.}\label{fig:clustergraphs}
\end{figure}

\begin{table}[htbp]
\begin{center}
\begin{tabular}{|c|| c|c|c|c|c|}
\hline
& $\lambda_1$&$\lambda_{N-1}$&$\kappa$&$\lambda_{N-1}-\lambda_1$\\\hline
$G$ & 0.2984 & 6.7016 & 22.4555 &  6.4031\\
$\tilde G_\kappa$&1.0000 & 17.9443 & {\bf 17.9443}&16.9443\\
$\tilde G_g$ & 0.0504 & 1.1542 & 22.8794 & {\bf 1.1038}\\\hline 
\end{tabular}
\end{center}
\caption{Comparison of condition number and spectral gap of the barbell graph, $G$, shown in Figure \ref{fig:completeclustercondition} and its rescaled versions, respectively.}\label{tab:clusterdata}
\end{table}

\end{example}

As discussed in the motivation of this section, reducing the condition number of a graph makes the graph more ``complete", that is, more like the complete graph in terms of its spectrum.  Since the algebraic connectivity $\lambda_1$ is as great as possible, it is the only graph for which $\lambda_1=\lambda_{N-1}$, the graph is the most connected a graph can possibly be, and as such the distance between any two points is minimal.  As previously discussed, the effective resistance is a natrual metric on graphs and one can compute that for any two distinct vertices, $i$ and $j$, on the complete graph on $N$ vertices we have
\begin{eqnarray*}
R(i,j)&=&\Sum_{k=1}^{N-1} \frac{1}{\lambda_k}\left( f_k(i)- f_k(j)\right)^2=\frac{1}{N}\Sum_{k=1}^{N-1}\left( f_k(i)- f_k(j)\right)^2\\
&=&\frac{1}{N}(e_i-e_j)^*  F F^* (e_i-e_j)=\frac{1}{N}(e_i-e_j)^*(e_i-e_j)\\
&=&\frac{1}{N}\norm{e_i-e_j}^2=\frac{2}{N}.
\end{eqnarray*}

\begin{conjecture}\label{conj:condresistance}
The process of conditioning a graph reduces the average resistance between any two vertices on the graph.
\end{conjecture}
The intuition behind Conjecture \ref{conj:condresistance} can be motivated by studying the quantity $\sum_{k=1}^{N-1}1/{\lambda_k}$.  Consider a sequence of positive numbers $\{a_k\}_{k=1}^{N}$ with average $\bar a=1/N\sum_{k=1}^N a_k$.  Then since the function $h(t)=1/t$ is continous and convex on the set of positive numbers, it is also midpoint convex on that set, i.e.,
$$\frac{N}{\bar a}=N h(\bar a) \leq \Sum_{k=1}^N h(a_k) = \Sum_{k=1}^N \frac{1}{a_k}.$$
With this fact, let $\{\lambda_k\}_{k=1}^{N-1}$ denote the eigenvalues of connected graph $G$ and $\{\tilde\lambda_k\}_{k=0}^{N-1}$ denote the eigenvalues of the conditioned graph $\tilde G_\kappa$, both satisfying $\bar\lambda=1/N\sum_{k=1}^{N-1} \lambda_k=1/N\sum_{k=1}^{N-1} \tilde\lambda_k$.  Since $\tilde G_\kappa$ is better conditioned than $G$, then $\norm{\sum_{k=1}^{N-1}\tilde\lambda_k-\bar\lambda}\leq\norm{\sum_{k=1}^{N-1}\lambda_k-\bar\lambda}$.  In other words, the eigenvalues $\{\tilde\lambda_k\}_{k=1}^{N-1}$ are closer to the average $\bar\lambda$ than the eigenvaleus $\{\lambda_k\}_{k=1}^{N-1}$ are.  Hence 
\begin{equation}\Sum_{k=1}^{N-1}\frac{1}{\tilde\lambda_k} \leq\Sum_{k=1}^{N-1}\frac{1}{\lambda_k}.\label{suminveig}\end{equation}
Equation \eqref{suminveig} almost resembles the effective resistance $R(i,j)=\sum_{k=1}^{N-1}1/\lambda_k ( f_k(i)- f_k(j))^2$ except for the term $( f_k(i)- f_k(j))^2$.  This term will be difficult to account for since little is known about the eigenvectors of $\tilde G_\kappa$.  Analysis of eigenvectors of perturbed matrices is a widely open area of research and results are very limited, see \cite{Katobook, Tao2012, stewart1990matrix, CloningerThesis}.

We remark that Conjecture \ref{conj:condresistance} claims that conditoning a graph will reduce the average effective resistance between points; it is not true that the resistance between all points will be reduced.  If the weight on edge $(i,j)$ is reduced, then its effective resistance between points $i$ and $j$ is increased.  Since we impose that the trace of the Laplacians be preserved, if any edge weights are increased, then by conservation at least one other edge's weight must be decreased.  The vertex pairs for those edges will then have an increased effective resistance between them. 

While we lack the theoretical justification, numerical simulations support Conjecture \ref{conj:condresistance} and this is a source of future work.

The authors of \cite{boyd2008} approach a similar way.  They propose using convex optimization to minimize the total effective resistance of the graph,
$$R_{tot}=\Sum_{i,j=1}^N R(i,j).$$
They show that the optimization problem is related to the problem of reweighting edges to maximize the algebraic connectivity $\lambda_1$.

\section*{Acknowledgment}
Radu Balan was partially supported by NSF grant DMS-1413249 and ARO grant W911NF1610008. Matthew Begu\'e and Chae Clark would like to thank the Norbert Wiener Center
for Harmonic Analysis and Applications for its support during
this research. 
Kasso Okoudjou was partially supported by a grant from the Simons Foundation ($\# 319197$ to Kasso Okoudjou), and ARO grant W911NF1610008.

\bibliographystyle{amsplain}
\bibliography{BBCO_arx.bib}
\end{document}